\documentclass[12pt]{article}
\usepackage{amsfonts,amssymb,amsmath,bm,a4wide}

\newtheorem{theorem}{Theorem}
\newtheorem{corollary}[theorem]{Corollary}
\newtheorem{lemma}[theorem]{Lemma}
\newtheorem{OP}{Open Problem}
\newtheorem{definition}{Definition}
\newtheorem{remark}[theorem]{Remark}

\newtheorem{problem}[theorem]{Problem}

\newtheorem{proposition}[theorem]{Proposition}
\newtheorem{notation}[theorem]{Notation}

\def\Q{\mathbb{Q}}
\newcommand{\QQ}{\mathbb{Q}}
\def\Z{\mathbb{Z}}
\newcommand{\ZZ}{\mathbb{Z}}
\def\N{\mathbb{N}}
\newcommand{\NN}{\mathbb{N}}

\def\cR{\mathcal{R}}

\newcommand{\vv}[1]{{\mathbf{#1}}}

\newenvironment{proof}{\noindent\textsc{{Proof}}.}{\hfill\raisebox{-1ex}{$\boxtimes$}}

\renewcommand{\tilde}{\widetilde}

\newcommand{\hidden}[1]{}

\newcommand{\ul}[1]{\underline{#1}}

\newcommand{\ftm}{f_{TM}}
\newcommand{\ftmm}{\tilde{f}_{TM}}
\newcommand{\cTM}{\tau_{TM}}

\newcommand{\tP}{\tilde{P}}
\newcommand{\tQ}{\tilde{Q}}

\title{Thue-Morse constant is not badly approximable}

\author{Dzmitry Badziahin\footnote{Durham University, Department of Mathematical Sciences, Science Laboratories, South Rd, Durham, DH1 3LE, United Kingdom} and Evgeniy Zorin\footnote{University of York, Department of Mathematics, York, YO10 5DD, United Kingdom}}

\begin{document}

\maketitle

\begin{abstract}
We prove that Thue-Morse constant $\cTM=0.01101001\dots_2$ is not a
badly approximable number. Moreover, we prove that $\cTM(a)=0.01101001\dots_a$ is not badly approximable for every integer base $a\geq 2$ such that $a$ is not divisible by 15. At the same time we provide a precise
formula for convergents of the Laurent series $\ftmm(z) =
z^{-1}\prod_{n=1}^\infty (1-z^{-2^n})$, thus developing further the research initiated by Alf van der Poorten and others.
\end{abstract}

\section{Introduction}
Let $\vv t=(t_0,t_1,\dots)=(0,1,1,0,1,0,0,\dots)$ be the Thue-Morse
sequence, that is the sequence $(t_n)_{n\in\N_0}$, where
$\N_0:=\N\cup\{0\}$, defined by recurrence relations $t_0=0$ and for
all $n\in\N_0$
$$
\begin{aligned}
t_{2n}&=t_n,\\
t_{2n+1}&=1-t_n.
\end{aligned}
$$
Thue-Morse sequence appears naturally in the description of many recurrent processes~\cite{AS2003}. It is often considered as the simplest non-trivial example of so called automatic sequences~\cite{AS2003}, i.e., sequences generated by finite automata (which are, in simple words, Turing Machines without memory tape).

Allouche and Shallit asked the following question (see~\cite{AS2003}, Open Problem~9, p.~403).
\begin{problem} \label{intro_main_pb}
Determine whether the partial quotients of the Thue-Morse constant
$\cTM$, defined by
\begin{equation} \label{defTM}
\cTM:=\sum_{k=0}^{\infty}\frac{t_k}{2^{k+1}},
\end{equation}
are bounded from above by a universal constant.
\end{problem}
\begin{remark}
The name of the constant~$\cTM$ defined by~\eqref{defTM} varies
slightly from one reference to another. In some sources it is called
Prouhet-Thue-Morse constant, and in some others it is referred to as
Thue-Morse-Mahler constant~\cite{BQ2013}. In this article we choose
the name Thue-Morse as the shortest commonly used name for it.
\end{remark}

Problem~\ref{intro_main_pb} can be easily translated into the language of
Diophantine approximations. It is a well known fact that the number
$x$ has bounded partial quotients if and only if it is badly
approximable, i.e., there is a constant $c>0$ such that for any
$p/q\in\Q$ we have
\begin{equation} \label{intro_pb_diophantine}
\left|x-\frac{p}{q}\right|>\frac{c}{q^2}
\end{equation}
(for instance, see~\cite{C1957}, Chapter~1, \S2). So
Problem~\ref{intro_main_pb} is equivalent to the question whether
the Thue-Morse constant $x = \cTM$ is badly approximable or in other
words whether it satisfies \eqref{intro_pb_diophantine}.

This problem attracted interest in the last years. In particular,
Bugeaud~\cite{B2011} showed that the transcendence exponent of
$\cTM$ is 2, that is for any $\varepsilon>0$ there exists a constant
$c_{\varepsilon}$ such that
\begin{equation} \label{intro_BR}
\left|\cTM-\frac{p}{q}\right|>\frac{c_{\varepsilon}}{q^{2+\varepsilon}}.
\end{equation}
Later Bugeaud and Queff\'elec~\cite{BQ2013} proved that the sequence
of partial quotients of $\cTM$ contains infinitely many values equal
to 4 or 5 and at the same time it contains infinitely many values
bigger than $50$. Note that in terms of the
inequality~\eqref{intro_pb_diophantine} this result implies an
absolute upper bound on the constant $c$: if such positive $c$
exists then $c<1/50$.

In this paper we solve Problem~\ref{intro_main_pb} by showing that
$\cTM$ is not a badly approximable number. This result is proved in
Theorem~\ref{theo_main_one} below.
To establish it, we provide a sequence
$\left(p_n/q_n\right)_{n\in\N}\in\QQ$ with $\lim_{n\to \infty}q_n\to
\infty$, of good rational approximations to the Thue-Morse constant.
These approximations satisfy
\begin{equation} \label{intro_good_approximations}
q_n\left|q_n\cTM-p_n\right|\rightarrow 0, \quad \text{ as } n\rightarrow\infty.
\end{equation}
It straightforwardly implies that the
inequality~\eqref{intro_pb_diophantine} is not satisfied for any
positive constant~$c$ and approximations $p_n/q_n$ where $n$ is
large enough.


We construct a sequence of approximations $\left(p_n/q_n\right)_{n\in\N}$ by a specialization of good \emph{functional} approximations to a so called Thue-Morse generating function:
\begin{equation} \label{ftm_presentation}
\ftm(z):=\sum_{i=0}^{\infty}(-1)^{t_i}z^i.
\end{equation}
It is easy to see that the Thue-Morse constant can be represented as
$$
\cTM=\frac12\left(1-\frac12\ftm(1/2)\right). 
$$
In our article we focus on a slightly modified version of this
function:
$$
\ftmm(z):=\frac{1}{z}\ftm(1/z).
$$
It is a Laurent series and one can easily check that $\cTM$ is a
badly approximable number if and only if $\ftmm(2)$ is badly
approximable too.

Note that the value $\ftmm(2)$ also appears in the work of
Dubickas~\cite{dub2006}. It is shown there that for every irrational
$x$ one has
$$
\inf_{n\in\NN}||2^n x||\le \ftmm(2)
$$
where $||\cdot||$ denotes the distance to the nearest integer.
Moreover for $x = \ftmm(2)$ the inequality becomes an equality.

To study functional approximations to $\ftmm(z)$ we apply the theory
of continued fractions for Laurent series, which is analogous
to the classical theory of continued fractions of rational numbers
(for instance, see~\cite{vdP}).

In our proof we take advantage of a functional equation for the function $\ftmm$ (see~\eqref{func_eq_2} below). This functional equation allows us, given one functional convergent to $\ftmm(z)$, to produce an
infinite sequence $p_n/q_n$ of rational approximations to $\ftmm(2)$
all satisfying
\begin{equation} \label{intro_sufficiently_good_approximation}
\left|\ftmm(2)-\frac{p_n}{q_n}\right|\leq\frac{C}{q_n^2},
\end{equation}
where the constant $C$ depends only on the initial functional approximation.

We find this construction interesting even on its own, as not only it allows
to reproduce the results from~\cite{BQ2013} by choosing a good
initial functional convergent to $\ftm$, but also it explains regularly
situated large partial quotients in the continued fraction of $\cTM$
of the same value which can be observed numerically. For instance,
in this way one can find an infinite sequence of partial quotients
equal to 2569. With some computational efforts one can check that it
is generated by the 15th convergent $P(z)/Q(z)$ of $\ftmm(z)$.

Next, we manage to use arguments on congruences and primitive roots
modulo $3^k$, $k\in\N$, to justify that in a carefully chosen
sequence of $p_n/q_n$
satisfying~\eqref{intro_sufficiently_good_approximation} there will
be arbitrarily large common factors $r_n$, hence after reducing by
this common factor the couple of integers $(p_n/r_n,q_n/r_n)$
verifies~\eqref{intro_good_approximations} and so $\ftmm(2)$ is not
badly approximable.

In this paper we also provide the precise formulae for computing the
convergents of Laurent power series $\ftmm(z)$. Our interest in this
subject is inspired by several papers by van der Poorten and others,
where they study continued fractions for functions given by infinite
products~\cite{AMvdP,vdP_first,vdP}. For instance, they numerically verified
that the first partial quotients of $\ftm(1/z)$ have degree at most
two. At the same time, the partial quotients of degree one have quickly growing coefficients, which is the generic behavior
(see~\cite{vdP}, section~2.1). The authors of~\cite{AMvdP} proved that all the partial quotients of Laurent power
series
$$
\sum_{k=0}^{\infty}(-1)^{t_k}z^{-3^k}
$$
have degree one.

It appears that the continued fraction for $\ftmm(z)$ is especially nice looking. Not only all its partial quotiens have degree one, but, moreover, all the even ones are rational multiples of $z-1$ and all the odd ones are rational multiples of $z+1$ (see Proposition~\ref{prop_part_q}). We even provide simple recurrent formulae allowing to calculate rational factors of $z-1$ and $z+1$, thus giving the exact values of all the partial quotients of $\ftmm$, see Proposition~\ref{proposition_coefficients_CF}. 


\paragraph{Generalizations.} The natural question is whether it is possible to generalize our considerations to a broader framework. The launching site for our constructions in this article is the functional equation~\eqref{func_eq_2}, so a natural extension is the following question.
\begin{OP} \label{OP1}
Let $d\in\N$, $d\geq 2$ and let the function $f(z)\in\QQ[[z]]$
satisfies the functional equation
\begin{equation} \label{fe_od}
f(z^d)=a(z)f(z)+b(z),
\end{equation}
where $a(z),b(z)\in\Q(z)$. Moreover assume that $f(z)$ is a transcendental function. Let $a\in\Q$ be a non-zero rational within the radius of convergence of $f$. Determine the conditions for $f(a)$ to be a badly approximable number.
\end{OP}
Our arguments in this article can be used to show that $\ftmm(a)$ is not badly approximable for all $a\in\N$, $n\geq 2$, with possible exceptions when $a$ is divisible by 15. This follows from Theorem~\ref{th_tma} (see Corollary~\ref{cor_tma}).

Note however that Van der Poorten and Shallit showed in~\cite{vdP_S} that the function $f_M(z)=\sum_{k=0}^{\infty}z^{2^k}$,  which satisfies the functional equation
$$
f_M(z^2)=f_M(z)-z,
$$
has a badly approximable value at $z=1/2$, moreover the continued fraction of $f_M(1/2)$ consists of just partial quotients 1 and 2 (actually in~\cite{vdP_S} this result is proved even for a much more general case of series $2\sum_{k=0}^{\infty}\pm 2^{-2^k}$).
So, the answer to Open Problem~\ref{OP1} definitively requires some additional conditions on the functional equation~\eqref{fe_od}, separating the case of badly approximable values from not badly approximable ones.

While  Open Problem~\ref{OP1} itself seems already enigmatic, we can consider even broader framework. The equation~~\eqref{fe_od} is a classical example of so called Mahler's functional equation. In the most general framework, the following system of functional equations is known as \emph{Mahler's system}:
\begin{equation} \label{systeme_1}
    a(\b{z})\ul{f}(z^d)=A(z)\ul{f}(z)+B(z),
\end{equation}
where $d\geq 2$ is an integer, $\ul{f}(z)=(f_1(z),\dots,f_n(z))\in\Q[[z]]^n$, $a(z)\in\Q[z]$, $A$ (resp. $B$) is an $n\times n$ (resp. $n\times 1$) matrix with coefficients in $\Q[z]$.
\begin{OP}
Let $d\in\N$, $d\geq 2$ and let the system of functions
$\ul{f}(z)=(f_1(z),\dots,f_n(z))\in(\QQ[[z]])^n$ be a solution to
the system~\eqref{systeme_1}. Moreover assume that $f_1(z)$ is a
transcendental function. For a non-zero rational $a\in\Q$ within the
radius of convergence of $\ul{f}(z)$ decide whether $f_1(a)$ is a
badly approximable number or not.
\end{OP}

\section{General facts}
It is well known (\cite{AS2003},  \S13.4) that the function $\ftm(z)$, defined by~\eqref{ftm_presentation}, admits the following presentation:
$$
\ftm(z)=\prod_{k=0}^{\infty}\left(1-z^{2^k}\right),
$$
and the following functional equation holds:
\begin{equation} \label{func_eq}
\ftm(z^2)=\frac{\ftm(z)}{1-z}.
\end{equation}

As we have mentioned in the introduction, we will focus on the study of a slightly modified version of $\ftm$:
$$
\ftmm(z):=\frac{1}{z}\ftm(1/z).
$$
Substituting $1/z$ in place of $z$ into~\eqref{func_eq} we find that $\ftmm$ satisfies the
following functional equation:
\begin{equation} \label{func_eq_2}
\ftmm(z^2)=\frac{\ftmm(z)}{z-1}.
\end{equation}

Easy verification shows that $\ftm(z)$ and $\ftmm(z)$ are closely
linked with Thue-Morse constant by
\begin{equation} \label{TM_link}
\cTM=\frac12\left(1-\frac12\ftm(1/2)\right)=\frac12\left(1-\ftmm(2)\right).
\end{equation}

By rewriting~\eqref{ftm_presentation} for $\ftmm(z)$ one can easily
check that $\ftmm(z)\in \QQ[[z^{-1}]]$. Moreover, all the
coefficients of the resulting Laurent series are either 1 or $-1$,
therefore it converges for $|z|>1$. Another consequence is that
$\ftmm(z)$ has the following continued fraction expansion 
$$
\ftmm(z) =[0;a_1(z), a_2(z),\ldots]=\dfrac{1}{a_1(z)+\dfrac{1}{a_2(z)+\dots}},
$$
where $a_i(z)\in\Q[z]$, $i\in\N$  (the details can be found in~\cite{vdP}).

The important consequences of this fact are that the
convergents $P_n(z) / Q_n(z)$ can be computed by the following
recurrent formulae
\begin{equation}\label{eq_pqn}
\begin{array}{rcl}
P_{n+1}(z) &=& a_{n+1}(z) P_n(z) + P_{n-1}(z),\\[1ex]
Q_{n+1}(z) &=& a_{n+1}(z) Q_n(z) + Q_{n-1}(z),
\end{array}
\end{equation}
for $n\ge 1$. Moreover Proposition 1 from~\cite{vdP} implies the following.
\begin{proposition} \label{proposition_convergent_iff}
Let $P(z),Q(z)\in\Q[z]$ be two polynomials. Then  $P(z) / Q(z)$ is a convergent to $\ftmm(z)$ if and only if
\begin{equation}\label{eq_convergent}
\deg(Q(z)\ftmm(z) - P(z)) < -\deg Q(z),
\end{equation}
where the degree of Laurent series
$G(z)=\sum_{k=h}^{\infty}a_kz^{-k}$, $a_h\ne 0$ is minus the smallest
index of a non-zero coefficient, that is in our notation we have
$\deg G=-h$.
\end{proposition}

Note that unlike the classical setup of rational numbers, where the numerators and
denominators $p_n$ and $q_n$ of convergents are defined uniquely,
$P_n(z)$ and $Q_n(z)$ are only unique up to multiplication by a
non-zero constant. At the same time, the polynomials computed by
formulae~\eqref{eq_pqn} in general are not monic. So, we can add a condition that the numerator $P_n$ has to be monic, producing a unique representative for each functional convergent to $\ftmm$.

These canonical representatives $\hat{P}_n(z)/\hat{Q}_n(z)$, with $\hat{P}_n$ a monic polynomial,  are still linked by recurrent relations similar to~\eqref{eq_pqn}:
\begin{equation}\label{eq_pqn2}
\begin{array}{rcl}
\hat{P}_{n+1}(z) &=& \hat{a}_{n+1}(z) \hat{P}_n(z) + \beta_{n+1}\cdot \hat{P}_{n-1}(z);\\[1ex]
\hat{Q}_{n+1}(z) &=& \hat{a}_{n+1}(z) \hat{Q}_n(z) +
\beta_{n+1}\cdot \hat{Q}_{n-1}(z)
\end{array}
\end{equation}
where we define, with $ \rho_n$ denoting the leading coefficient of $P_n$,
$$
\hat{a}_{n+1}(z) = \frac{a_{n+1}(z)\cdot \rho_n}{\rho_{n+1}}\;\mbox{ and } \;
\beta_{n+1} = \frac{\rho_{n-1}}{\rho_{n+1}}.
$$
One can easily check from~\eqref{eq_pqn2} that $\hat{a}_n(z)$ are
always monic, while the denominators $\hat{Q}_n$ may be not monic in a general case. However we will see in the next section that for the function $\ftmm$ both numerators and denominators of the canonical representatives $\hat{P}_n(z)/\hat{Q}_n(z)$ of convergents are monic polynomials.

\section{Continued fraction of the function $\ftmm$}


\begin{lemma}\label{lem_pqz2}
Let $P(z) / Q(z)$ be a convergent to $\ftmm(z)$. Then $P^*(z) /
Q^*(z)$ is also a convergent, where
\begin{equation}\label{eq_pqz2}
P^*(z) = (z-1)P(z^2);\quad Q^*(z) = Q(z^2).
\end{equation}
\end{lemma}
\begin{proof} It is a consequence of the functional relation for
$\ftmm(z)$. Indeed,
$$
\deg (Q^*(z)\ftmm(z) - P^*(z)) = \deg((z-1)(Q(z^2)\ftmm(z^2) -
P(z^2)))\le -2\deg(Q)-1.
$$
At the same time, $\deg Q^*(z)=2\deg Q(z)$. Therefore $Q^*(z)$ and $P^*(z)$ satisfy the condition~\eqref{eq_convergent}. We conclude that $P^*(z) /
Q^*(z)$ is a convergent to $\ftmm$ by Proposition~\ref{proposition_convergent_iff}.
\end{proof}

Recursive application of Lemma~\ref{lem_pqz2} enables us to
construct an infinite sequence of convergents to $\ftmm(z)$ starting
from only one convergent. However not every convergent can
be constructed in this way. For example, by a direct computation one
can find the convergents
\begin{equation} \label{first_convergents}
\frac{1}{1+z}\quad \mbox{and}\quad \frac{z^2 - 2}{z^3+z^2}.
\end{equation}
Then \eqref{eq_pqz2} immediately gives us the convergents
\begin{equation} \label{second_convergents}
\frac{z-1}{z^2+1}, \quad \frac{(z-1)(z^2-1)}{z^4+1},\;\mbox{and}\;
\frac{(z-1)(z^4-2)}{z^6+z^4}.
\end{equation}
However these calculations, using  Lemma~\ref{lem_pqz2} and the
initial convergents~\eqref{first_convergents} only, neither allow
one to construct the convergent $P(z)/Q(z)$ to $\ftmm(z)$ with
$\deg(Q) = 5$, nor give an information whether such a convergent
exists. The next proposition shows that, in fact, for every $n\in\N$
there is a convergent $P_n(z)/Q_n(z)$ to $\ftmm(z)$ such that
$\deg(Q_n) = n$.
\begin{proposition}\label{prop_part_q}
Let $[0;a_1(z), a_2(z),\ldots,]$ be the continued fraction expansion of
$\ftmm(z)$. Then $\forall n\ge 2$,
$a_{2n-1}(z)=\alpha_{2n-1}\cdot(z+1)$ and
$a_{2n}(z)=\alpha_{2n}\cdot (z-1)$ where $\alpha_i\in\QQ$, $i\in\N$, are
constants. Moreover for every $n\in\NN$, the convergent $Q_{2n}(z)$ is of the form
$Q_{2n}(z) = Q_n(z^2)$, in particular it is an even function; $Q_{2n-1}(z)$ is of the
form $Q_{2n-1}(z) = (z+1)\cdot Q^+_{n-1}(z^2)$, where $Q^+_{n-1}(X)$ is a polynomial of degree $n-1$ with rational coefficients.
\end{proposition}

\begin{proof} We reason by induction. We already checked this statement for $n=2$. Now
we need to check it for $2n+1$ and $2n+2$ given that the statement
for $3,4,\ldots, 2n$ is true.

Note that induction hypothesis and~\eqref{eq_pqn} imply that $\deg Q_k=k$ for $k=1,2,3,\dots,2n$. In particular, the convergent $Q_{n+1}$ has degree $n+1$.

By Lemma~\ref{lem_pqz2} we have that $Q^*_{n+1}(z) = Q_{n+1}(z^2)$
is also a convergent and $\deg Q^*_{n+1}=2n+2$. Moreover we have $\deg Q_{2n}=2n$, thus
$Q^*_{n+1}(z)$ has to coincide either with $Q_{2n+1}(z)$ or with
$Q_{2n+2}(z)$. The first case is actually impossible, because
otherwise we would have had, using induction hypothesis
and~\eqref{eq_pqn},
\begin{multline} \label{prop_part_q_case_one_2n+1}
Q^*_{n+1}(z)=Q_{n+1}(z^2) = Q_{2n+1}(z) = a_{2n+1}(z)\cdot Q_{2n}(z) + Q_{2n-1}(z) \\ = a_{2n+1}(z)\cdot Q_{n}(z^2) + (z+1)Q_{n-1}^+(z^2).
\end{multline}
To show that the equality~\eqref{prop_part_q_case_one_2n+1} is unattainable, substitute $-z$ in place of $z$ to~\eqref{prop_part_q_case_one_2n+1} and then apply $(-z)^2=z^2$. We find
\begin{equation} \label{prop_part_q_case_one_2n+1_modified}
Q_{n+1}(z^2)=a_{2n+1}(-z)\cdot Q_{n}(z^2) + (1-z)Q_{n-1}^+(z^2).
\end{equation}
Subtracting~\eqref{prop_part_q_case_one_2n+1_modified} from~\eqref{prop_part_q_case_one_2n+1} and dividing by 2 we obtain
\begin{equation} \label{prop_part_q_case_one_2n+1_modified_further}
\frac{a_{2n+1}(-z)-a_{2n+1}(z)}{2}\cdot Q_{n}(z^2) = zQ_{n-1}^+(z^2).
\end{equation}
The equality~\eqref{prop_part_q_case_one_2n+1_modified_further} is impossible, because its right hand side is a non-zero polynomial of degree 2n-1, and the left hand side is either a zero or a polynomial of degree at least $2n$. This contradiction shows that $Q^*_{n+1}(z)$ cannot coincide with $Q_{2n+1}(z)$, thus we have
$$
Q^*_{n+1}(z)=Q_{2n+2}(z).
$$
In particular we see that $\deg Q_{2n+2}=2n+2$, thus $\deg Q_{2n+1}=2n+1$ and so
\begin{equation} \label{deg_a_i_is_one}
\deg a_{2n+1}=\deg a_{2n+2}=1.
\end{equation}
Applying again the induction hypothesis and the
formulae~\eqref{eq_pqn} for convergents $Q_{2n+1}$ and $Q_{2n+2}$ we
have
\begin{align*}
Q_{2n+2}(z) &= Q_{n+1}(z^2) = a_{2n+2}(z)\cdot Q_{2n+1}(z)+Q_{2n}(z)
\\&=a_{2n+2}(z)\cdot a_{2n+1}(z)\cdot Q_{2n}(z)+a_{2n+2}(z)\cdot Q_{2n-1}(z)+Q_{2n}(z).
\\&=(a_{2n+2}(z)\cdot a_{2n+1}(z)+1)\cdot Q_{n}(z^2) + a_{2n+2}(z)\cdot (z+1)Q^+_{n-1}(z^2).
\end{align*}
Again, by substituting $-z$ in place of $z$ and then using $(-z)^2=z^2$ we deduce
\begin{multline} \label{eqn_ccc}
\frac{a_{2n+2}(z)\cdot a_{2n+1}(z)-a_{2n+2}(-z)\cdot a_{2n+1}(-z)}{2} Q_{n}(z^2)
\\=\frac{a_{2n+2}(-z)\cdot (-z+1)-a_{2n+2}(z)\cdot (z+1)}{2}Q_{n-1}(z^2)
\end{multline}
Both polynomials $a_{2n+2}(z)\cdot a_{2n+1}(z)-a_{2n+2}(-z)\cdot a_{2n+1}(-z)$ and $a_{2n+2}(-z)\cdot (-z+1)-a_{2n+2}(z)\cdot (z+1)$ are odd functions, so they are either 0 or of degree 1 (taking into account~\eqref{deg_a_i_is_one}). Therefore if the left hand side of~\eqref{eqn_ccc} is not zero then it has degree $2n+1$, and the right hand side in this case has degree $2n-1$. This is a contradiction, so we conclude that both sides of~\eqref{eqn_ccc} are 0. This gives us
\begin{align*}
a_{2n+2}(-z)\cdot (-z+1)&=a_{2n+2}(z)\cdot (z+1),\\
a_{2n+2}(z)\cdot a_{2n+1}(z)&=a_{2n+2}(-z)\cdot a_{2n+1}(-z).
\end{align*}

This system implies that $a_{2n+2}(z)$ is a
multiple of $z-1$ and $a_{2n+1}$ is a multiple of $z+1$. Now the fact that $Q_{2n+1}(z) = (z-1)Q^+_{n}(z^2)$ readily follows from~\eqref{eq_pqn}. This concludes the proof.
\end{proof}

To compute the convergents of $\ftmm(z)$ we just need to find the
precise values of coefficients $\alpha_n$ such that
$a_n(z)=\alpha_n(z-(-1)^n)$. To this end, it appears easier to make calculations with canonical representatives of the functional convergents, described at the end of the previous section, that is with the fractions $\hat{P}_n(z)/\hat{Q}_n(z)$ such that $\hat{P}_n(z)$ is a monic polynomial.

As we mentioned in the previous section, the
formulae~\eqref{eq_pqn} do not always (in fact, almost never)
produce monic polynomials $P_n(z)$ and $Q_n(z)$. For example, one
can check that if we start with $Q_1(z)=1+z$ and $Q_2(z)=z^2+1$ then
$a_3(z)=-z-1$ and therefore
$$
Q_3(z) = -(z+1)Q_2(z) + Q_1(z) = -z^3 - z^2.
$$
Moreover further calculations show that $P_n(z)$ and $Q_n(z)$ do not
always have integer coefficients.


So, as we are interested in monic numerator, in our case formulae~\eqref{eq_pqn2}
together with Proposition~\ref{prop_part_q} give
\begin{align} \label{formula_Q_n_new}
\hat{Q}_{n+1}(z) &= (z + (-1)^n)\hat{Q}_n(z) + \beta_{n+1}\hat{Q}_{n-1}(z),\\
\hat{P}_{n+1}(z) &= (z + (-1)^n)\hat{P}_n(z) +
\beta_{n+1}\hat{P}_{n-1}(z), \label{formula_P_n_new}
\end{align}
where $\beta_{n+1}=\frac{\alpha_n}{\alpha_{n+1}}$, $n\in\N$, $n\geq
2$. Polynomials $\hat{P}_n(z)$ and $\hat{Q}_n(z)$ are linked to the
original polynomials $P_n(z)$ and $Q_n(z)$ by
$\hat{Q}_n(z):=\frac{Q_n(z)}{\prod_{k=1}^n\alpha_k}$ and
$\hat{P}_n(z):=\frac{P_n(z)}{\prod_{k=1}^n\alpha_k}$ and one readily
verifies that they both are monic for all $n\in\N$.

\begin{proposition} \label{proposition_coefficients_CF}
The coefficients $\beta_n$ in~\eqref{formula_Q_n_new} and~\eqref{formula_P_n_new} can be computed recursively by the
following formulae
\begin{align}
&\beta_3 = -1,\; \beta_4 = 1,\\
&\beta_{2n+1} = \label{proposition_coefficients_CF_formula_two}
-\frac{\beta_{n+1}}{\beta_{2n}},\;\\
&\beta_{2n+2} = 1+(-1)^n-
\beta_{2n+1} \label{proposition_coefficients_CF_formula_three}
\end{align}
for every positive integer $n\ge 2$.
\end{proposition}

\begin{proof} The values $\beta_3$ and $\beta_4$ can be computed directly
from already known $Q_1(z), Q_2(z),$ $Q_3(z)$ and $Q_4(z)$ (see~\eqref{first_convergents} and~\eqref{second_convergents}). Next, we
substitute $z\mapsto z^2$ into the formula~\eqref{formula_Q_n_new} for $\hat{Q}_{n+1}(z)$ and use
Proposition~\ref{prop_part_q} 
to get that
$$
\hat{Q}_{2n+2}(z) = (z^2 +(-1)^n)\hat{Q}_{2n}(z) + \beta_{n+1} \hat{Q}_{2n-2}(z).
$$
On the other hand, by directly applying~\eqref{formula_Q_n_new} we
have
\begin{eqnarray*}
 \hat{Q}_{2n+2}(z) &=& (z-1)\hat{Q}_{2n+1}(z) + \beta_{2n+2}\hat{Q}_{2n}(z)\\
&=& (z-1)\left((z+1)\hat{Q}_{2n}(z) + \beta_{2n+1} \hat{Q}_{2n-1}(z)\right) +
\beta_{2n+2}Q_{2n}(z).
\end{eqnarray*}
By comparing these two formulae for $\hat{Q}_{2n+2}(z)$ we get the equation
\begin{equation} \label{proposition_coefficients_CF_formula_one}
(\beta_{2n+2}-1-(-1)^n)\hat{Q}_{2n}(z) + \beta_{2n+1}\cdot
(z-1)\hat{Q}_{2n-1}(z) - \beta_{n+1}\hat{Q}_{2n-2}(z) = 0.
\end{equation}
By looking at the coefficient of $z^{2n}$ we get
$\beta_{2n+2}-1-(-1)^n + \beta_{2n+1} = 0$, which proves~\eqref{proposition_coefficients_CF_formula_three}. The formula~\eqref{proposition_coefficients_CF_formula_two} is
achieved by substituting Formula~\eqref{formula_Q_n_new} for $\hat{Q}_{2n}(z)$ into Equation~\eqref{proposition_coefficients_CF_formula_one} and looking at the coefficient of $z^{2n-2}$:
\begin{eqnarray*}
-\beta_{2n+1}\cdot((z-1)\hat{Q}_{2n-1}(z) + \beta_{2n}\hat{Q}_{2n-2}(z)) +
\beta_{2n+1}\cdot (z-1)\hat{Q}_{2n-1}(z) - \beta_n\hat{Q}_{2n-2}(z) = 0.
\end{eqnarray*}
It readily follows $\beta_{2n+1}\beta_{2n} +\beta_{n+1} = 0$.
\end{proof}

Now we have precise recursive formulae to quickly compute
convergents to $\ftmm$ as far as we want. For example, the first few
convergents following $P_4(z) / Q_4(z)$ are
$$
\frac{z^4 - z^2 - 1}{(z+1)(z^4+z^2+1)},\; \frac{(z-1)(z^4-2)}{z^6 +
z^4},\; \frac{z^6-2z^4-z^2+3}{(z+1)(z^6-z^2-1)},\ldots
$$
The 9th convergent is of particular value for us so we write it
down as well:
\begin{equation} \label{ninth_convergent}
\hat{P}_9(z):=z^8-3z^6+2z^4+3z^2-4,\quad \hat{Q}_9(z):=(z+1)(z^8-z^6+z^2+2).
\end{equation}
This convergent plays the central role in our proof that $\cTM$ is not badly approximable.

\section{Rational approximations to the Thue-Morse constants.}

We start by extracting a specific subsequence of convergents to $\ftmm$.
\begin{definition} \label{def_tP_tQ}
Let $n\in\N$. Define
$$
\tP_{n}(z):=\prod_{k=0}^{n}(z^{2^k}-1)\hat{P}_9(z^{2^{n+1}})
$$
and
$$
\tQ_{n}(z):=\hat{Q}_9(z^{2^{n+1}}),
$$
where $\hat{P}_9(z)/\hat{Q}_9(z)$ is the 9th functional convergent
to $\ftmm$ given at the end of the previous section
(see~\eqref{ninth_convergent}).
\end{definition}

\begin{remark}
{\rm Iteratively applying Lemma~\ref{lem_pqz2} we find that
$\tP_n(z) = \hat{P}_{9\cdot 2^{n+1}}(z)$ and $\tQ_n(z) =
\hat{Q}_{9\cdot 2^{n+1}}(z)$. In particular, $\tP_n(z)/\tQ_n(z)$ is
indeed a convergent to $\ftmm(z)$.}
\end{remark}


\begin{lemma} \label{lemma_rational_approximation_n}
For every $n\in\NN$ and $z_0\in\N$, $z_0\geq 2$ the integers $\tP_n(z_0)$ and
$\tQ_n(z_0)$ satisfy the following Diophantine approximation
properties:
\begin{eqnarray}
\left|\ftmm(z_0)-\frac{\tilde{P}_n(z_0)}{\tilde{Q}_n(z_0)}\right|&\leq& C\cdot z_0^{-36 \cdot 2^n}, \label{ie_convergent_numeric} \\
\tilde{Q}_{n}(z_0)&\leq&2\cdot z_0^{18\cdot 2^{n}},\label{ie_deQ_numeric}
\end{eqnarray}
where the constant $C = C(z_0)$ is independent of $n$.
\end{lemma}
\begin{proof}
Consider the following function:
\begin{equation} \label{def_F}
F(z):=\ftmm(z)-\frac{\hat{P}_9(z)}{\hat{Q}_9(z)}.
\end{equation}
As $\hat{P}_9(z)/\hat{Q}_9(z)$ is the 9th convergent to $\ftmm(z)$, it follows from Proposition~\ref{proposition_convergent_iff} that $\deg(F)\leq -19$.

In fact, one can check by an explicit calculation that $F(z)$, being an infinite series in $\frac{1}{z}$, starts from the term $\frac{6}{z^{19}}$, that is
\begin{equation} \label{eq_F_order_of_approximation}
F(z)=\frac{6}{z^{19}}+O\left(\frac{1}{z^{20}}\right).
\end{equation}
Further, consider $F(z^{2^{n+1}})\prod_{k=1}^{n}(z^{2^k}-1)$ where $n\in\N$. By using the
definition~\eqref{def_F} of $F$, functional
equation~\eqref{func_eq_2} and
estimate~\eqref{eq_F_order_of_approximation} we find
\begin{multline} \label{lemma_rational_approximation_n_one}
F(z^{2^{n+1}})\prod_{k=1}^{n}(z^{2^k}-1)=\ftmm(z)-\frac{\hat{P}_9(z^{2^{n+1}})\prod_{k=1}^{n}(z^{2^k}-1)}{\hat{Q}_9(z^{2^{n+1}})}\\=\frac{6\prod_{k=1}^{n}(z^{2^k}-1)}{z^{19\cdot
2^{n+1}}}+O\left(\frac{\prod_{k=1}^{n}(z^{2^k}-1)}{z^{20\cdot
2^{n+1}}}\right),
\end{multline}
where the constant implied by the symbol $O(\cdot)$ is independent of $n$.

By Definition~\ref{def_tP_tQ},
$$
\ftmm(z)-\frac{\hat{P}_9(z^{2^{n+1}})\prod_{k=1}^{n}(z^{2^k}-1)}{\hat{Q}_9(z^{2^{n+1}})}=\ftmm(z)-\frac{\tilde{P}_{n}(z)}{\tilde{Q}_{n}(z)}.
$$
At the same time, we have $\prod_{k=1}^{n}(z^{2^k}-1)<z^{2^{n+1}}$ for any $z>1$, hence we infer from~\eqref{lemma_rational_approximation_n_one}
$$
\ftmm(z)-\frac{\tilde{P}_{n}(z)}{\tilde{Q}_{n}(z)}=\frac{6}{z^{18\cdot 2^{n+1}}}+O\left(\frac{1}{z^{19\cdot
2^{n+1}}}\right)=\frac{6}{z^{36\cdot 2^{n}}}+O\left(\frac{1}{z^{38\cdot
2^{n}}}\right)
$$
and~\eqref{ie_convergent_numeric} follows.

Further, a straightforward calculation shows that
$\hat{Q}_9(z_0)\le 2z_0^9$ for every $z_0\geq 2$. Then
$$
\tilde{Q}_{n}(z_0):=\hat{Q}_9(z_0^{2^{n+1}})\leq 2\cdot z_0^{9\cdot
2^{n+1}}=2\cdot z_0^{18\cdot 2^{n}},
$$
which proves~\eqref{ie_deQ_numeric}.
\end{proof}

This lemma shows that for $z_0\in\NN$, $z_0\geq 2$,
$\tP_n(z_0)/\tQ_n(z_0)$ already provides sufficiently good rational
approximation to $\ftmm(z_0)$ (following the scheme presented in the
introduction, they provide approximations
satisfying~\eqref{intro_sufficiently_good_approximation}). In order
to show that $\ftmm(z_0)$ is not badly approximable it is sufficient
for every $r>1$ to find an $n\in\NN$ such that both $\tP_n(z_0)$ and
$\tQ_n(z_0)$ have a common factor bigger than $r$. We prove this
fact in Lemma~\ref{crucial_simplification}.

In further discussion we will stick to the case $z_0=2$, however as
we will see in the next section similar ideas should work for other
positive integers $z_0$.

The following chain of simple lemmas prepares the proof of our
essential ingredient, Lemma~\ref{crucial_simplification}. We start
with the following classical result which can be found for instance
in ~\cite{NZM1991}[p.~102].
\begin{lemma} \label{lemma_primitive_root}
Let $p$ be an odd prime and $g$ be a primitive root modulo $p^2$. Then $g$ is a primitive
root modulo $p^m$ for all $m\in\N$.
\end{lemma}

One can check that $2$ is a primitive root modulo $3^2$ therefore
a direct corollary of this Lemma is that $2$ is a primitive root
modulo $3^m$ for every $m\in\NN$.


\begin{lemma} \label{lemma_first_congruence}
Let $m\in\N$ and let $t$ be an even integer, $t\not\equiv 0\pmod 3$.
Then there exists $n\in\N$ such that
$$
2^n\equiv t \mod 2\cdot 3^m.
$$
Moreover, one can choose such $n$ to verify additionally $n\leq 2\cdot 3^{m-1}$.
\end{lemma}
\begin{proof}
As $t$ is even, there exists $k\in\N$ such that $t=2k$. Since $2$ is
a primitive root modulo $3^m$ and $k$ is coprime to $3$, there
exists $n\in\N$ such that
\begin{equation} \label{lemma_first_congruence_one}
2^{n-1}\equiv k\mod 3^m,
\end{equation}
where $k$ is as above. Multiplying the congruence~\eqref{lemma_first_congruence_one} by $2$ we find
$$
2^{n}\equiv t\mod 2\cdot 3^m,
$$
hence the claim.

To prove the concluding part of the lemma, note that the size of multiplicative group of residues modulo $3^m$ is $2\cdot 3^{m-1}$. So in~\eqref{lemma_first_congruence_one} we can always choose $n$ verifying $n-1\leq 2\cdot 3^{m-1}-1$, and the second claim of the lemma follows.
\end{proof}

In the next lemma we use the following notation.
\begin{notation} \label{a_db_b}
Let $a,b\in\Z$. We write $a\,||\, b$ if $a$ divides $b$, but $a^2$ does not.
\end{notation}
\begin{lemma} \label{lemma_main_congruence}
Let $m\in\N$, $m\geq 2$ and let $t$ be an integer such that $3\,||\, t-1$. Then
there exists $n\in\N$ such that
$$
2^{2^n}\equiv t \mod 3^m.
$$
Moreover, one can choose such $n$ to verify additionally $n\leq 2\cdot 3^{m-2}$.
\end{lemma}
\begin{proof}
Since 2 is a primitive root modulo $3^m$ there is a $k\in\N$ such
that
\begin{equation} \label{lemma_main_congruence_first_congruence}
2^k\equiv t \mod 3^m.
\end{equation}
Reducing congruence~\eqref{lemma_main_congruence_first_congruence} modulo 3 we find (using our assumption on $t$)
$$
2^k\equiv 1 \mod 3,
$$
hence $k$ is an even positive integer. Furthermore, by
reducing~\eqref{lemma_main_congruence_first_congruence} modulo 9, we
get $2^k\not\equiv 1\pmod 9$ therefore $k$ is not a multiple of 3.

By Lemma~\ref{lemma_first_congruence} we have that there exists an $n\in\N$ such that
$$
2^n\equiv k\mod 2\cdot 3^{m-1}
$$
and $n\leq 2\cdot 3^{m-2}$.

At the same time, by Euler's theorem we have
$$
2^{2\cdot 3^{m-1}}\equiv 1 \mod 3^m,
$$
as $\phi(3^m)=3^m-3^{m-1}=2\cdot 3^{m-1}$. Therefore
$$
2^{2^n}\equiv 2^k \mod 3^m,
$$
and we conclude by comparing this last congruence with the congruence~\eqref{lemma_main_congruence_first_congruence}.
\end{proof}

\begin{lemma} \label{crucial_simplification}
For any $m\in\N$, $m\geq 3$ there exists an index $n_m$ such that both integers
$\tilde{P}_{n_m}(2)$ and $\tilde{Q}_{n_m}(2)$ are divisible by
$3^m$, and moreover $n_m\leq 3^{m-1}$.
\end{lemma}
\begin{proof}
We verify by a direct calculation that $Q_9(1)=6$ is divisible by 3
but not by 9, and
$$
Q_9'(1)=11\not\equiv 0 \mod 3.
$$
So by Hensel's lemma, for every $m\in\N$, $m\geq 1$, there exists a solution $x_m\in\N$ to the congruence
$$
Q_9(x_m)\equiv 0 \mod 3^m
$$
such that this solution $x_m$ is congruent to $1$ modulo $3$ and
$x_m\not\equiv 1\pmod 9$. By Lemma~\ref{lemma_main_congruence},
there exists $t_m\in\N$ such that $2^{2^{t_m}}\equiv x_m\mod 3^m$,
thus
\begin{equation} \label{tQ_tm_conclusion}
\tQ_{n_m}(2)=Q_9(2^{2^{t_m}})\equiv 0 \mod 3^m.
\end{equation}
Moreover, by the same Lemma we can choose $t_m$ to satisfy $t_m\leq 2\cdot 3^{m-2}$.

One can easily check that if  $t_m$ is a solution to the
equation~\eqref{tQ_tm_conclusion} then every integer $t>1$ such that
$t\equiv t_m\pmod{\phi(2\cdot 3^{m-1})}$ is also a solution. So for
any $l\in\N$ a number $t_m+l\cdot 2\cdot 3^{m-2}$ also provides a
solution to~\eqref{tQ_tm_conclusion}.

If $t_m>m$ then we choose $n_m:=t_m$, otherwise $n_m:=t_m+2\cdot
3^{m-2}$. Note that for $m\geq 3$ we have $3^{m-2}\geq m$, thus our
definition of $n_m$ assures that for $m\ge 3$ we have $n_m>m$ and
$n_m\leq 3^{m-1}$.


Further, note that for any $k\in\N$ we have $2^{2^k}\equiv 1 \mod 3$. Therefore
$$
\prod_{k=0}^{n_m-1}(2^{2^k}-1)\equiv 0 \mod 3^{n_m-1}
$$
and we readily have that
$$
\tilde{P}_{n_m}(2):=\prod_{k=0}^{n_m-1}(2^{2^k}-1)P_9(2^{2^{n_m+1}})\equiv
0 \mod 3^{n_m-1}.
$$
Finally, we infer from $n_m>m$ that
\begin{equation} \label{tP_tm_conclusion}
\tP_{n_m}(2)\equiv 0\mod 3^m.
\end{equation}
Congruences~\eqref{tQ_tm_conclusion} and~\eqref{tP_tm_conclusion}
show that $n_m$ indeed verifies the properties claimed in the
statement of the lemma, and this completes the proof.
\end{proof}

\begin{theorem} \label{theo_main_one}
Thue-Morse constant $\tau_{TM}$ is not badly approximable. Moreover,
there exists a constant $c>0$ such that the inequality
\begin{equation} \label{strong_ie}
\left|\tau_{TM}-p/q\right|\le \frac{c}{q(\log\log q)^2}.
\end{equation}
has infinitely many solutions $(p,q)\in\N^2$.
\end{theorem}
\begin{proof}
We are going to prove that an analogue of~\eqref{strong_ie} is
satisfied for $\ftmm(2)$. Then the relation~\eqref{TM_link} would
straightforwardly imply the same condition on $\cTM$ too.

By Lemma~\ref{lemma_rational_approximation_n} we have that the sequence $p_n:=\tilde{P}_{n}(2)$ and $q_n:=\tilde{Q}_{n}(2)$ provide sufficiently good rational approximations to $\ftmm(2)$, that is for any $n$ we have
$$
\left|\ftmm(2)-\frac{p_n}{q_n}\right|\leq\frac{C}{q^2_n},
$$
where the constant $C$ is independent of $n$.
Moreover, by Lemma~\ref{crucial_simplification} we have a
subsequence of indices $(n_m)_{m\in\N}$ such that both integers
$\tilde{P}_{n_m}(2)$ and $\tilde{Q}_{n_m}(2)$ are divisible by
$3^m$. Therefore the integers
$\tilde{p}_{n_m}:=\frac{\tilde{P}_{n_m}(2)}{3^m}$ and
$\tilde{q}_{n_m}:=\frac{\tilde{Q}_{n_m}(2)}{3^m}$ satisfy
\begin{equation} \label{main_theo_proof_ie}
\left|\ftmm(2)-\frac{\tilde{p}_{n_m}}{\tilde{q}_{n_m}}\right|\leq\frac{C}{3^{2m}\tilde{q}_{n_m}^2},
\end{equation}
which readily implies that the number $\ftmm(2)$ is not badly approximable.

To justify~\eqref{strong_ie}, note that because of the bound $n_m\le
3^{m-1}$ and the explicit formula
$$
\tilde{Q}_{n_m}(2):=Q_9(2^{2^{n_m+1}}),
$$
there exists a constant $c_1$ independent of $m$ such that
$\tilde{q}_{n_m}\leq c_1 2^{9\cdot 2^{3^m}}$, so for $m$
sufficiently large we have $\log\log\tilde{q}_{n_m}\leq 2\cdot 3^m$
and~\eqref{strong_ie} follows from~\eqref{main_theo_proof_ie}.
\end{proof}

\section{Constants $\ftmm(a)$ for arbitrary $a\in\NN$}

The proposed chain of lemmata suggests the method for checking
whether the value $\ftmm(a)$ is badly approximable for an arbitrary $a\in\NN$,
$a>1$. We formulate it as the following theorem.

\begin{theorem}\label{th_tma}
Assume that there exist positive integers $n,t,p$ such that
\begin{enumerate}
\item $p$ is a prime such that $p\,||\, a^{2^n}-1$ (recall Notation~\ref{a_db_b});
\item $2$ is a primitive root modulo $p^2$;
\item $p\,||\, \hat{Q}_t(1)$;
\item $\hat{Q}'_t(1)\not\equiv 0\pmod p$.
\end{enumerate}
Then $\ftmm(a)$ is not badly approximable. Moreover,
there exists a constant $c>0$ such that the inequality
\begin{equation} \label{strong_ie_general}
\left|\ftmm(a)-p/q\right|\le \frac{c}{q(\log\log q)^2}.
\end{equation}
has infinitely many solutions $(p,q)\in\N^2$.
\end{theorem}


\begin{proof}
For any $n,t\in\NN$ we define
$$
\tP_{n,t}(z):=\prod_{k=0}^{n}(z^{2^k}-1)\hat{P}_t(z^{2^{n+1}})
$$
and
$$
\tQ_{n,t}(z):=\hat{Q}_t(z^{2^{n+1}}).
$$
Then the same arguments as in
Lemma~\ref{lemma_rational_approximation_n} imply that for every
$n\in\NN$ and $z_0\in\NN$ with $z_0\ge 2$ one has
\begin{eqnarray}
\left|\ftmm(z_0)-\frac{\tilde{P}_{n,t}(z_0)}{\tilde{Q}_{n,t}(z_0)}\right|&\leq& C\cdot z_0^{-4t \cdot 2^n}, \\
\tilde{Q}_{n,t}(z_0)&\leq&C\cdot z_0^{2t\cdot 2^{n}},
\end{eqnarray}
where the constant $C = C(z_0,t)$ is independent of $n$.

Values $\tilde{P}_{n,t}(a)$ and $\tilde{Q}_{n,t}(a)$ are not
necessarily integer. However $\tilde{P}_{n,t}(a)\in 1/d_P\cdot \ZZ$
and $\tilde{Q}_{n,t}(a) \in 1/d_Q\cdot \ZZ$ where $d_P$
(respectively $d_Q$) is the least common multiple of all
denominators of the rational coefficients of the polynomial
$\hat{P}_t(z)$ (respectively of $\hat{Q}_t(z)$). So for every pair
of integers $p_n,q_n$ where $p_n = d_pd_Q\cdot\tilde{P}_{n,t}(a)$ ,
$q_n = d_pd_Q\cdot\tilde{Q}_{n,t}(a)$ the following inequality takes
place:
$$
\left|\ftmm(a) - \frac{p_n}{q_n}\right|\le \frac{d_Q^2\cdot
C^3}{q_n^2}.
$$
Note that the value $d_Q^2C^3$ depends only on $a$ and $t$ and does
not depend on $n$. Therefore it is enough to find arbitrarily large
$r\in\ZZ$ and some $n\in\NN$ such that values $p_n$ and $q_n$ have a
common factor $r$. We will show that positive integer power of $p$ can play the role of such common factor $r$.

By Condition 1 of the theorem, $p\mid a^{2^n}-1$. Therefore for
every $m>n$ we have $p^{m-n}\mid \tilde{P}_{m,t}(a)$.

Conditions 3 and 4 and Hensel's lemma imply that the equation
$\hat{Q}_t(x) = 0$ has a solution $x\in\Z_p$ such that $x\equiv
1\pmod p$ and $x\not\equiv 1\pmod{p^2}$. Next, since 2 is a
primitive root modulo $p^2$ (in view of condition 2), then by
Lemma~\ref{lemma_primitive_root} it is also a primitive root modulo
every power~$p^m$, $m\in\N$.

For every $m\in\N$, the multiplicative group $\cR^*_{p^m}$ of
residues modulo $p^m$ has the order $\phi(p^m)=(p-1)p^{m-1}$. As the
element $a^{2^n}$ is congruent to 1 modulo $p$, it lies in the
kernel of the canonical projection
$\cR^*_{p^m}\rightarrow\cR^*_{p}$. The multiplicative group
$\cR^*_{p}$ of residues modulo $p$ has the order $p-1$, so the
residue $a^{2^n}$ has the order $p^{l}$ in $\cR^*_{p^m}$, for some
$l\leq m-1$. If the value $l$ is strictly smaller than $m-1$, then
we necessarily have $a^{2^n}\equiv 1\mod p^2$, which contradicts the
Condition 1, hence the multiplicative order of $a^{2^n}$ modulo
$p^{m}$ is exactly $p^{m-1}$ and thus the set of residues
$\{a^{2^n\cdot s}\mod p^m\!: s\in \NN, \gcd(s,p)=1\}$ coincides with
the set of residues modulo $p^m$ congruent to 1 modulo $p$ but not
congruent to 1 modulo $p^2$. So, there is an $s\in\N$ such that
$a^{2^n\cdot s}\equiv x \mod p^{m}$ and $s\not\equiv 0\mod p$.

As $2$ is a primitive root modulo $p^{m-1}$ and $a^{2^n}$ has order
$p^{m-1}$ modulo $p^m$, we have that the set of residues
$\{a^{2^n\cdot 2^s}\mod p^m\!\!: s\in \NN\}$ coincides with the set
of residues $\{a^{2^n\cdot s}\mod p^m\!: s\in \NN, \gcd(s,p)=1\}$.
In particular, there exists $s_1$ such that $a^{2^{n+s_1}}\equiv
x\mod p^m$.


Moreover, as $s_1$ is defined modulo $\phi(p^{m-1})=(p-1)p^{m-2}$, one can choose such $s_1$ that $n_m:=n+s_1$ verifies
\begin{equation} \label{general_theo_size_ie}
m<n_m\leq m+(p-1)p^{m-2}.
\end{equation}
So we get that
$\tilde{Q}_{n_m,t}(a)$ is divisible by $p^m$ and therefore both
$p_{n_m}$ and $q_{n_m}$ have common divisor $p^{m-n}$. 
be taken arbitrary large this finishes the proof of the theorem.
We deduce that there exists a constant $C_1$ such that for any $m>n$
$$
\left|\ftmm(a) - \frac{\tilde{p}_m}{\tilde{q}_m}\right|\le \frac{C_1}{p^{2n_m}\tilde{q}_m^2},
$$
where $\tilde{p}_m=p_{n_m}/p^m$ and $\tilde{q}_m=q_{n_m}/p^m$ are integers. The upper inequality in~\eqref{general_theo_size_ie} implies $\log\log q_{n_m}\ll n_m\ll p^m$, so moreover we have $\log\log \tilde{q}_m\ll p^m$ and the inequality~\eqref{strong_ie_general} follows.
\end{proof}

{\bf Remark.} Conditions 2 -- 4 of Theorem~\ref{th_tma} do not
depend on $a$ at all. One can look at them as conditions on a prime
number $p$. We call $p$ {\it acceptable} if there exists $t\in\NN$
such that the Conditions 2 -- 4 are satisfied. Then
Theorem~\ref{th_tma} can be reformulated as follows: if there exist
$n\in\NN_0$ and an acceptable prime $p$ such that $p\,||\,
a^{2^n}-1$ then $\ftmm(a)$ is not badly approximable. By testing
various polynomials $\hat{Q}_t(z)$ it is easy to find many
acceptable primes. In the previous section we have already checked
that $3$ is acceptable. By considering
$$
\hat{Q}_{11}(z)=x^{11}+x^{10}+\frac{2 x^9}{3}+\frac{2 x^8}{3}+\frac{4 x^7}{3}+\frac{4 x^6}{3}+x^5+x^4+\frac{2 x^3}{3}+\frac{2 x^2}{3}+\frac{x}{3}+\frac{1}{3},
$$
one can check that $5$ is acceptable too. This remark already leads
us to the following corollary, which generalizes
Theorem~\ref{theo_main_one}.

\begin{corollary} \label{cor_tma}
Let $a\in\N$ be a positive integer which is not divisible by 15. Then $\ftmm(a)$ is not badly approximable.
\end{corollary}
\begin{proof}
As we have noted just before this corollary, primes 3 and 5 are
acceptable, that is they verify conditions 2--4 of
Theorem~\ref{th_tma}.

It is an easy exercise, which we leave to the reader, to check that for $a\not\equiv 0\pmod{15}$ either 3 or 5 satisfies the
condition 1 of Theorem~\ref{th_tma} and therefore by this theorem we obtain that $\ftmm(a)$ is not badly approximable.
\end{proof}

So the remaining uncovered case is $\ftmm(a)$ where $a$ is divisible
by 15.

Unfortunately, Theorem~\ref{th_tma} can not be applied to show that
$\ftmm(15)$ is not badly approximable. Indeed, the numbers $15-1,
15^2-1$ and $15^4-1$ have prime divisors $2,7$ and $113$, and 2 is
not a primitive root for neither of them. Other prime divisors of
$15^{2^n}-1$ for some $n\in\NN$ must also divide $15^{2^m}+1$ for
some $m\ge 2$. It is a classical result that such primes $p$ satisfy
the condition $p\equiv 1\pmod 8$. Since $2$ is a quadratic residue
modulo such primes $p$ then the condition 2 is never satisfied.

However Theorem~\ref{th_tma} can still work for $a$ equal to some of
the multiples of $15$. For example, for $a=30$ we have $29 \mid
30-1$ and 29 is an acceptable prime (one can take $t = 35$). Also,
$$
11\mid 45 -1;\quad 61\mid 60^2 - 1;\quad 19\mid 75^2 - 1;\quad
13\mid 90^2-1
$$
and primes $11, 61, 19, 13$ are acceptable. One may check this by
taking $t = 43, 49, 19$ and $33$ respectively.
So, all the values $a\in\N$ such that $2\leq a\leq 104$ and $a\ne
15$ assure that $\ftmm(a)$ is not badly approximable.

\section*{Acknowledgements}

The authors are grateful to Yann Bugeaud for valuable advices and discussions, which contributed a lot to the development of the paper.
The first author acknowledges the support of the EPSRC grant number EP/L005204/1.

\end{document}